\documentclass[letterpaper, 10pt, conference]{ieeeconf}  

\IEEEoverridecommandlockouts                              
\usepackage{hyperref}

\usepackage{graphicx}

\usepackage{amsmath}
\usepackage{amssymb}
\usepackage{amsfonts}
\usepackage{times}
\usepackage{mathptmx} 
\usepackage{datetime}

\newtheorem{thm}{Theorem}

\DeclareMathAlphabet{\bit}{OML}{cmm}{b}{it}
\def\<{\leqslant}           
\def\>{\geqslant}           

\def\d{\partial}
\def\wh{\widehat}

\def\Re{\mathrm{Re}}   
\def\Im{\mathrm{Im}}   

\def\mR{{\mathbb R}}    
\def\mC{\mathbb{C}}    

\def\Tr{\mathrm{Tr}}       
\def\rT{{\rm T}}        
\def\diam{\diamond}       

\def\bE{\mathbf{E}}    

\def\bra{{\langle}}
\def\ket{{\rangle}}

\def\re{{\rm e}}        
\def\rd{{\rm d}}        

\def\bJ{\mathbf{J}}

\def\br{\mathbf{r}}
\def\x{\times}
\def\ox{\otimes}

\def\fF{\mathfrak{F}}
\def\fR{\mathfrak{R}}

\def\fH{\mathfrak{H}}

\def\cF{\mathcal{F}}

\def\cK{\mathcal{K}}

\def\cI{\mathcal{I}}

\def\ups{\upsilon}
\def\Ups{\Upsilon}

\def\diag{\mathop{\rm diag}}    
\def\blockdiag{\mathop{\rm blockdiag}}    

\title{\LARGE \bf
A Quantum Karhunen-Loeve Expansion and Quadratic-Exponential Functionals for Linear Quantum Stochastic  Systems}

\author{Igor G. Vladimirov,\qquad
Ian R. Petersen,\qquad
Matthew R. James
\thanks{This work is supported by the Air Force Office of Scientific Research (AFOSR) under agreement number FA2386-16-1-4065 and the Australian Research Council under grant DP180101805.}
\thanks{The authors are with the
Research School of Electrical, Energy and Materials Engineering, College of Engineering and Computer Science,
Australian National University, Canberra, Acton, ACT 2601,
Australia, {\tt
igor.g.vladimirov@gmail.com, i.r.petersen@gmail.com, matthew.james@anu.edu.au.
}
}
}

\begin{document}

\maketitle

\begin{abstract}
This paper extends the Karhunen-Loeve representation from classical Gaussian random processes to quantum Wiener processes which model external bosonic fields for open quantum systems. The resulting expansion of the quantum Wiener process in the vacuum state is organised as a series of sinusoidal functions on a bounded time interval with statistically independent coefficients consisting of noncommuting position and momentum operators in a Gaussian quantum state. A  similar representation is obtained for the solution of a linear quantum stochastic differential equation which governs the system variables of an open quantum harmonic oscillator. This expansion  is applied to computing a quadratic-exponential functional arising as a performance criterion in the framework of risk-sensitive control for this class of open quantum systems.
\end{abstract}

\section{INTRODUCTION}

The Karhunen-Loeve (KL) representation \cite{GS_2004} provides a series expansion of a classical random process over a bounded time interval in terms of an orthonormal basis of deterministic functions with random coefficients. The basis is usually formed from the eigenfunctions of a self-adjoint integral operator whose kernel is the covariance function of the process, in which case, the resulting coefficients  are uncorrelated (and hence, independent for Gaussian processes). Similarly to the Ritz-Galerkin methods \cite{M_1950}, the KL approach employs the idea of meshless approximation of continuous time functions (as opposed to the time discretization in finite-difference schemes for numerical solution of ordinary  differential equations).

The covariance kernel of the standard Wiener process \cite{KS_1991} has sinusoidal eigenfunctions, which makes its KL representation  with such a basis particularly suitable  for the solution of a linear stochastic differential equation (SDE) driven by the Wiener process. The action of a linear input-output operator, associated with the SDE, reduces to a linear transformation of the random coefficients of the KL expansion \cite{GS_2004}.

The present paper extends this idea to linear quantum SDEs (QSDEs), which are driven by quantum Wiener processes whose role in the Hudson-Parthasarathy  quantum stochastic calculus \cite{HP_1984,P_1992,P_2015} (not only in regard to linear QSDEs) is similar to that of the standard Wiener process in the classical case. The quantum Wiener process on a symmetric Fock space \cite{PS_1972} represents bosonic quantum fields (such as quantised electromagnetic radiation), and the QSDEs model the interaction of open quantum systems with such fields. Both the system and field variables are time-varying operators on a system-field tensor-product Hilbert space, and their evolution is specified by the system Hamiltonian and system-field coupling operators. 
 Because of the noncommutative nature of these quantum variables, their statistical properties are described in quantum probabilistic terms \cite{H_2018,M_1995} which do not reduce to classical joint probability distributions.

The quantum KL (QKL) representation,  which is considered here for the quantum Wiener process, inherits the sinusoidal basis    from its classical predecessor. However,  the coefficients of the QKL expansion consist of noncommuting operators which satisfy the canonical commutation relations (CCRs) of the quantum mechanical  position and momentum operators \cite{S_1994}. We apply this QKL expension to the system variables of an open quantum harmonic oscillator (OQHO) (with a quadratic Hamiltonian and linear coupling),
governed by a linear QSDE, which constitutes a building  block of linear quantum systems theory \cite{NY_2017,P_2017}.
A more natural QKL expansion for the system variables is also obtained  by using the orthonormal eigenfunctions of the two-point quantum covariance kernel for the invariant Gaussian state  of the stable OQHO with vacuum input fields.

We then outline a combination of the QKL expansion of the system variables with symplectic techniques   \cite{VPJ_2018c} in application to computing a quadratic-exponential functional (QEF) \cite{VPJ_2018a}.
The QEF is an alternative (though closely related \cite{VPJ_2019}) version of the original quantum risk-sensitive cost \cite{J_2004,J_2005}. Its minimization (by an appropriate choice of a quantum controller or filter for a given quantum plant) makes the closed-loop system more conservative in the sense of large deviations of quantum trajectories \cite{VPJ_2018a} and more robust to quantum statistical uncertainties described in terms of quantum relative entropy \cite{OW_2010,YB_2009}   with respect to the nominal system-field state \cite{VPJ_2018b}.
These properties  of the QEF make its computation an important robust performance analysis problem in addition to the fact that similar problems arise 
in regard to the characteristic (or moment-generating) functions for quadratic Hamiltonians \cite{PS_2015} and the quantum L\'{e}vy area \cite{CH_2013,H_2018}.

The paper is organised as follows.
Section~\ref{sec:QKLW} develops a QKL representation for a multichannel quantum Wiener process.
Section~\ref{sec:stat} considers the statistical properties of the QKL coefficients when the quantum Wiener process is in the vacuum state.
Section~\ref{sec:Xsin} obtains a sinusoidal representation for the system variables of an OQHO as the solution of a linear QSDE driven by the quantum Wiener process in the QKL form.
Section~\ref{sec:QKLX} develops a QKL expansion for the system variables of the OQHO using their invariant multipoint Gaussian quantum state. 
Section~\ref{sec:QEF} applies the QKL representation to computing the QEF for linear quantum stochastic systems.
Section~\ref{sec:conc} provides concluding remarks.

\section{KARHUNEN-LOEVE REPRESENTATION OF QUANTUM WIENER PROCESSES}
\label{sec:QKLW}

Let $W:=(W_k)_{1\< k \< m}$ be a multichannel quantum Wiener process, organised as a column-vector  of an even number of self-adjoint operators $W_1(t), \ldots, W_m(t)$ on a symmetric Fock space $\fF$ \cite{P_1992}, which depend on time $t\>0$ and represent bosonic fields. In accordance with its continuous tensor-product structure \cite{PS_1972}, $\fF$ is endowed with a filtration in the form of an increasing family of subspaces $\fF_t$, so that $W_k(t)$ acts effectively on $\fF_t$ for any $t\>0$ and $k=1, \ldots, m$. The component  quantum Wiener processes satisfy the two-point CCRs
\begin{align}
\nonumber
    [W(s), W(t)^\rT]
    & := ([W_j(s), W_k(t)])_{1\< j,k\< m}\\
\label{WWcomm}
    & =
    2i\min(s,t)J,
    \qquad
    s,t\>0.
\end{align}
Here, $[\alpha,\beta]:= \alpha\beta - \beta\alpha$ is the commutator of linear operators,
\begin{equation}
\label{JJ}
    J:=  \bJ \ox I_{m/2}
\end{equation}
is an orthogonal real  antisymmetric matrix (so that $J^2 = -I_m$), where $\ox$ is the Kronecker product, $I_m$ is the identity matrix of order $m$, and
\begin{equation}
\label{bJ}
\bJ: = {\begin{bmatrix}
        0 & 1\\
        -1 & 0
    \end{bmatrix}}
\end{equation}
spans the subspace of antisymmetric matrices of order 2.  If $q$ and $p:= -i\d_q$ are the quantum mechanical position and momentum operators acting on the Schwartz space \cite{V_2002}, the vector $v:= {\scriptsize \begin{bmatrix}q\\ p\end{bmatrix}}$ has the CCR matrix $\frac{1}{2}\bJ$ in the sense that $[v, v^\rT] = i\bJ$ (in view of $[q,p] = i$). Therefore, if $(q_k, p_k)$, with $k=1, \ldots, \frac{m}{2}$, are conjugate position-momentum pairs on an appropriate tensor-product Hilbert space, then
the vector $r:= \sqrt{2}[q_1, \ldots, q_{m/2}, p_1, \ldots, p_{m/2}]^\rT$ of $m$ self-adjoint operators satisfies $[r, r^\rT] = 2iJ$ with the same matrix $J$ as in (\ref{WWcomm}), (\ref{JJ}).

The two-point CCR structure (\ref{WWcomm}) of the continuous-time quantum Wiener process $W$ can be achieved by using an auxiliary sequence of pairwise commuting vectors of quantum variables. More precisely, let $w_0,w_1,w_2,\ldots$ be vectors of $m$ self-adjoint operators  on a complex separable Hilbert space $\cF$
satisfying  the CCRs
    \begin{equation}
    \label{wwcomm}
      [w_j,w_k^\rT]
      =
      2i\delta_{jk}J,
      \qquad
      j,k=0,1,2,\ldots,
    \end{equation}
where $\delta_{jk}$ is the Kronecker delta. In particular, the commutativity  between the entries of $w_j$, $w_k$ for all $j\ne k$ holds when the entries of $w_0,w_1,w_2,\ldots$ are defined on
different Hilbert spaces $\cF_0, \cF_1, \cF_2, \ldots$, respectively  (which can be copies of a common Hilbert space) and are extended to the infinite-tensor-product space $\cF:= \bigotimes_{k=0}^{+\infty} \cF_k$.

Now, for a fixed but otherwise arbitrary time horizon $T>0$, consider the eigenfunctions
\begin{equation}
\label{feig}
    f_k(t)
    :=
    \sqrt{\tfrac{2}{T}}
    \sin(\omega_k t),
    \quad
    \omega_k
    :=
    \tfrac{\pi}{T} (k+\tfrac{1}{2}),
    \qquad
    k=0,1, 2,\ldots
\end{equation}
of the integral operator whose kernel is the covariance function of the standard Wiener process:
\begin{equation}
\label{eig}
    \int_0^T \min(s,t) f_k(t) \rd t
    =
    \lambda_k f_k(s),
    \qquad
    0\< s\< T,
\end{equation}
 with the eigenvalues related to the frequencies $\omega_k$ by \cite[p. 229]{GS_2004}
\begin{equation}
\label{lamom}
    \lambda_k = \tfrac{1}{\omega_k^2}.
\end{equation}
The eigenfunctions (\ref{feig}) are orthonormal in the Hilbert space $L^2([0,T])$ of square integrable functions on the time interval $[0,T]$:
\begin{equation}
\label{ffeig}
    \bra
        f_j,
        f_k
    \ket
    :=
    \int_0^T
    f_j(t)f_k(t)\rd t = \delta_{jk},
    \qquad
    j,k=0,1,2,\ldots.
\end{equation}
The kernel function in (\ref{eig}) is represented by an absolutely and uniformly convergent series
\begin{equation}
\label{minst}
    \min(s,t)
    =
    \sum_{k=0}^{+\infty}
    \lambda_k
    f_k(s)f_k(t),
    \qquad
    0\< s,t \< T,
\end{equation}
with $\sum_{k=0}^{+\infty} \lambda_k = \int_0^T t\rd t = \frac{1}{2}T^2$. 
Similarly to the KL representation of the standard Wiener process \cite[Eq. (16) on p. 229]{GS_2004}, consider its quantum counterpart
\begin{align}
\nonumber
    W(t)
    & =
    \cI
    +
    \sum_{k=0}^{+\infty}
    \sqrt{\lambda_k} f_k(t) w_k\\
\label{Weig}
    & =
    \cI+
    \sqrt{\tfrac{2}{T}}
    \sum_{k=0}^{+\infty}
    \tfrac{1}{\omega_k}
    \sin(\omega_k t)
    w_k,
\end{align}
where $\cI$ is a vector of $m$ copies of the identity operator on $\cF$. This infinite linear combination of the functions (\ref{feig}) (whose ``coefficients'' are the vectors $w_k$ with operator-valued entries) is
a vector of $m$ time-varying self-adjoint operators on $\cF$.
 Its two-point commutator matrix is computed\footnote{without using the particular sinusoidal structure of the eigenfunctions} by combining the bilinearity of the commutator with (\ref{wwcomm}), (\ref{minst}) as
\begin{align}
\nonumber
    [W(s),W(t)^\rT]
     & =
    \sum_{j,k=0}^{+\infty}
    \sqrt{\lambda_j \lambda_k}
    f_j(s)f_k(t) [w_j,w_k^\rT]\\
\nonumber
    & =
    2i
    \sum_{k=0}^{+\infty}
    \lambda_k
    f_k(s)f_k(t)
    J \\
\label{WWcomm1eig}
    & =
    2i\min(s,t) J,
    \qquad
    0\< s,t \< T,
\end{align}
which is identical to the commutation structure of the quantum Wiener process in (\ref{WWcomm}). In view of (\ref{lamom}), the orthonormality (\ref{ffeig}) allows the coefficients $w_k$ to be  recovered from $W$ in (\ref{Weig}) as
\begin{equation}
\label{wWeig}
    w_k
    =
    \omega_k
    \int_0^T
    f_k(t)
    (W(t)-\cI)
    \rd t,
    \qquad
    k = 0,1,2,\ldots.
\end{equation}
In fact, similarly to the classical case, by starting from the quantum Wiener process $W$ on the Fock space $\fF$ (so that the entries of the vectors $w_k$ in (\ref{wWeig}) are also defined on $\fF$), it follows that the CCRs (\ref{WWcomm}) lead to  (\ref{wwcomm}). Indeed, since the identity operator commutes with any operator, (\ref{WWcomm}), (\ref{eig})--(\ref{ffeig}) imply
\begin{align}
\nonumber
    [w_j, w_k^\rT]
    & =
    \omega_j
    \omega_k
    \int_{[0,T]^2}
    f_j(s)
    f_k(t)
    [W(s), W(t)^{\rT}]
    \rd s
    \rd t\\
\nonumber
    & =
    2i
    \omega_j
    \omega_k
    \int_{[0,T]^2}
    \min(s,t)
    f_j(s)
    f_k(t)
    \rd s
    \rd t
    J\\
\label{WWcomm2eig}
    & =
    2i
    \frac{\omega_j}{\omega_k}
    \bra
        f_j,
        f_k
    \ket
    J
    =
    2i
    \delta_{jk}
    J,
\end{align}
which reproduces the CCRs (\ref{wwcomm}). Therefore, in view of (\ref{WWcomm1eig}), (\ref{WWcomm2eig}),  in the framework of the expansion (\ref{Weig}),  the CCRs (\ref{WWcomm}) are equivalent to (\ref{wwcomm}), with both CCRs remaining valid regardless of a particular quantum state for the Wiener process. 

\section{STATISTICAL PROPERTIES OF THE COEFFICIENTS}
\label{sec:stat}

We will be concerned mainly with the case of fields in the vacuum state \cite{P_1992}. In terms of their behavior over the time interval $[0,T]$, this means that the quasi-characteristic functional (QCF) of the quantum Wiener process $W$  takes the form
\begin{equation}
\label{QCFeig}
    \bE \re^{i \int_0^T f(t)^\rT \rd W(t)}
    =
    \re^{-\frac{1}{2} \|f\|^2},
    \qquad
    f \in L^2([0,T], \mR^m),
\end{equation}
where $\bE \xi := \Tr(\rho \xi)$ is the expectation of a quantum variable $\xi$ over an underlying density operator $\rho$, and $\|f\|:= \sqrt{\int_0^T |f(t)|^2\rd t}$ is the $L^2$-norm for square integrable vector-valued functions on $[0,T]$.
\begin{thm}
\label{th:Omegaeig}
Suppose the quantum Wiener process $W$ is in the vacuum state in the sense of (\ref{QCFeig}). Then the vectors $w_k$ in (\ref{wWeig}) are statistically independent and are in a joint Gaussian quantum  state with zero mean and common covariance matrix
\begin{equation}
\label{Omegaeig}
    \Omega := I_m + iJ,
\end{equation}
with the matrix $J$ given by (\ref{JJ}), (\ref{bJ}),
so that
\begin{equation}
\label{EwEwweig}
    \bE w_k = 0,
    \quad
    \bE(w_jw_k^\rT) = \delta_{jk}\Omega,
    \qquad
    j,k=0,1,2,\ldots.
\end{equation}
\hfill$\square$
\end{thm}
\begin{proof}
In view of (\ref{feig}),  the integration by parts in (\ref{wWeig}) leads to
\begin{align}
\nonumber
    w_k
    & =
    \omega_k
    \sqrt{\tfrac{2}{T}}
    \int_0^T
    \sin(\omega_k t)
    (W(t)-\cI)
    \rd t\\
\nonumber
    & =
    \int_0^T
    (\cI-W(t))
    \rd
    g_k(t)\\
\label{wWeig1}
    & =
    \int_0^T
    g_k(t)
    \rd W(t),
\end{align}
where use is also made of the initial condition $W(0)=\cI$ together  with the functions
\begin{equation}
\label{gkeig}
    g_k(t)
    :=
    \sqrt{\tfrac{2}{T}}
    \cos(\omega_k t),
    \qquad
    k=0,1,2,\ldots,
\end{equation}
which satisfy $g_k(T) = \sqrt{\frac{2}{T}}\cos(\pi (k+\frac{1}{2})) = 0$ and  are also orthonormal in $L^2([0,T])$.
For any $N\>0$ and $u_0, \ldots, u_N \in \mR^m$, the joint QCF of the vectors $w_0, \ldots, w_N$ in (\ref{wWeig1}) is computed as
\begin{align}
\nonumber
    \bE \re^{i\sum_{k=0}^N u_k^\rT w_k}
    & =
    \bE \re^{i\sum_{k=0}^N u_k^\rT \int_0^T g_k(t)\rd W(t)}\\
\nonumber
    & =
    \bE \re^{i \int_0^T f(t)^\rT \rd W(t)}\\
\label{QCF1eig}
    & =
    \re^{-\frac{1}{2} \|f\|^2}.
\end{align}
Here, (\ref{QCFeig}) is used, and the function
$f: [0,T]\to \mR^m$ is given by
\begin{equation}
\label{ffueig}
    f(t) = \sum_{k=0}^N g_k(t)u_k.
\end{equation}
In view of the orthonormality of (\ref{gkeig}), it follows from (\ref{ffueig}) that
\begin{equation}
\label{fnormeig}
    \|f\|^2 = \sum_{j,k=0}^N \bra g_j,g_k\ket u_j^\rT u_k = \sum_{k=0}^N |u_k|^2.
\end{equation}
Substitution of (\ref{fnormeig}) into (\ref{QCF1eig}) shows that $w_0, \ldots, w_N$ are in a Gaussian quantum state \cite{KRP_2010} with zero mean and the joint covariance matrix $I_{N+1}\ox \Omega $, with $\Omega$ given by (\ref{Omegaeig}) in view of (\ref{wwcomm}). An equivalent form of these two moments is provided by (\ref{EwEwweig}).
\end{proof}

The matrix $\Omega$ in (\ref{Omegaeig}) is the Ito matrix of the quantum Wiener process $W$ in the sense that $\rd W \rd W^\rT = \Omega \rd t$.
A reasoning, similar to that in the proof of Theorem~\ref{th:Omegaeig}, shows that if $W$ is in a more general  Gaussian state, then so also are the vectors $w_0, w_1, w_2, \ldots$, except that the latter are no longer statistically independent.
Therefore, the representation (\ref{Weig}) relates the commutation structure and statistical properties of the quantum Wiener process $W$ with those of the coefficients $w_0, w_1, w_2, \ldots$. This representation is a quantum counterpart of the Karhunen-Loeve expansion \cite{GS_2004} of classical random processes.

\section{SINUSOIDAL EXPANSION FOR 
SOLUTIONS OF LINEAR QSDES}
\label{sec:Xsin}

Consider an OQHO, which interacts with external bosonic fields and is endowed with an even number of system variables $X_1(t), \ldots, X_n(t)$. These quantum variables are time-varying self-adjoint operators, acting on the system-field tensor-product space $\fH:= \fH_0 \ox \fF$ (where $\fH_0$ is a complex separable Hilbert space playing the role  of the initial system space for $X_1(0), \ldots, X_n(0)$). Also, the system variables satisfy the Weyl CCRs \cite{F_1989} whose infinitesimal Heisenberg form is given by
\begin{equation}
\label{XCCR}
    [X(t),X(t)^\rT] = 2i \Theta,
    \qquad
    X:=(X_k)_{1\< k\< n},
\end{equation}
for any $t\>0$,
where $\Theta$ is a constant real antisymmetric matrix of order $n$, which is assumed to be nonsingular.  The vector $X$ of the system variables evolves according to a linear QSDE
\begin{equation}
\label{dX}
    \rd X = AX \rd t + B\rd W
\end{equation}
(the time arguments are omitted for brevity),
driven by the quantum Wiener process $W$ of Section~\ref{sec:QKLW}. Here, the matrices $A \in \mR^{n\x n}$, $B\in \mR^{n\x m}$ satisfy the physical realizability (PR) condition \cite{JNP_2008}
\begin{equation}
\label{PR}
    A \Theta + \Theta A^\rT + BJB^\rT = 0,
\end{equation}
which is closely related to the preservation of the CCRs (\ref{XCCR}) in time. The property (\ref{PR}) follows from the parameterization of the matrices
\begin{equation}
\label{AB}
    A = 2\Theta (R + M^\rT JM),
     \qquad
     B = 2\Theta M^\rT
\end{equation}
in terms of the energy and coupling matrices $R = R^\rT \in \mR^{n\x n}$, $M \in \mR^{m\x n}$ which specify the system Hamiltonian $\frac{1}{2} X^\rT R X$ and the vector $MX$ of $m$ system-field coupling operators. Moreover, if the matrix $A$ in (\ref{AB}) is Hurwitz, then the CCR matrix is uniquely recovered as the solution $\Theta = \int_0^{+\infty} \re^{tA} BJB^\rT \re^{tA^\rT} \rd t$ of (\ref{PR}) as an algebraic Lyapunov equation (ALE).

In addition to the input fields $W_1, \ldots, W_m$ and the internal dynamic variables $X_1, \ldots, X_n$, the OQHO also has output field variables $Y_1, \ldots, Y_m$ whose evolution is affected  by the system-field interaction. However, the output fields will not be considered in what follows.

Now, due to linearity of the QSDE (\ref{dX}), its solution is given by
\begin{equation}
\label{Xsol}
    X(t)
    =
    \re^{tA} X_0 + \int_0^t \re^{(t-s)A} B \rd W(s),
    \qquad
    t\> 0,
\end{equation}
with its entries acting on the corresponding system-field subspace $\fH_t:= \fH_0\ox \fF_t$, where $X_0:= X(0)$ for brevity. Similarly to classical linear systems, the following Laplace transforms\footnote{note that the integrals in (\ref{XWhat}) have different structure}
\begin{equation}
\label{XWhat}
    \wh{X}(v)
    :=
    \int_0^{+\infty}
    \re^{-vt}
    X(t)
    \rd t,
    \quad
    \wh{W}(v)
    :=
    \int_0^{+\infty}
    \re^{-vt}
    \rd W(t),
\end{equation}
which are well-defined for any $v\in \mC$ satisfying $\Re v>0$ (with $A$ being Hurwitz), 
are related by
\begin{equation}
\label{XF}
    \wh{X}(v)
    =
    F(v)
    (B \wh{W}(v) +X_0),
    \qquad
    F(v):= (vI_n - A)^{-1}
\end{equation}
(see also \cite{YK_2003}). Here, $FB$ is the $\mC^{n\x m}$-valued transfer function  from $W$ to $X$, which is specified by the pair $(A,B)$. The matrices $F(u)$, $F(v) \in \mC^{n\x n}$ in (\ref{XF})
commute  with the matrix $A$ and with each other (as functions of a common matrix \cite{H_2008})  for any $u,v \in \mC$ which are not eigenvalues of $A$.

The following theorem establishes a representation for the system variables over the time interval $[0,T]$ by using the QKL expansion (\ref{Weig}) of the driving quantum Wiener process and making advantage of the sinusoidal nature of the eigenfunctions in (\ref{feig}). Its formulation employs auxiliary matrices
\begin{align}
\label{mhok}
  \mho_k
  & :=
    (\omega_k^2 I_n + A^2)^{-1},\\
\label{Ak}
  A_k
  & :=
      \sqrt{\tfrac{2}{T}}
        A
        \mho_k
        \big(
                (-1)^k
                \re^{TA}
                -
                \tfrac{1}{\omega_k}A
        \big)
\end{align}
(also commuting with each other and the matrix $A$), where $\omega_k$ are the frequencies from (\ref{feig}).

 \begin{thm}
   \label{th:Xfg}
   For the OQHO, described by (\ref{XCCR})--(\ref{AB}), with $A$ Hurwitz, the vector of the system variables can be represented as
   \begin{equation}
   \label{Xfg}
     X(t)
     =
     \xi
     +
     \sum_{k=0}^{+\infty}
     (f_k(t) \alpha_k + g_k(t)\beta_k),
     \qquad
     0\< t\< T.
   \end{equation}
   Here, the functions $f_k$, $g_k$ are given by (\ref{feig}), (\ref{gkeig}), and $\alpha_k$, $\beta_k$ are vectors of $n$ self-adjoint quantum variables which are related to the initial system variables in (\ref{Xsol}) and the QKL coefficients in (\ref{Weig}) by
   \begin{equation}
   \label{ab}
    \alpha_k
    :=
    A_k \xi + \omega_k \mho_k B w_k,
    \qquad
    \beta_k
    :=
     - A\mho_k Bw_k
   \end{equation}
   where $\xi$ is also such a vector given by
\begin{equation}
\label{xi}
    \xi
    :=
        X_0
        +
    \sqrt{\tfrac{2}{T}}
    \sum_{k=0}^{+\infty}
    A \mho_k
    B
    w_k,
\end{equation}
and use is also made of the matrices $\mho_k$, $A_k$ from (\ref{mhok}), (\ref{Ak}).
\hfill$\square$
 \end{thm}
 \begin{proof}
 By substituting (\ref{Weig}) into (\ref{Xsol}) and using an operator version of the complex impedance technique, it follows that
 \begin{align}
\nonumber
    X(t)
    & =
    \re^{tA}
    X_0
    +
    \sqrt{\tfrac{2}{T}}
    \sum_{k=0}^{+\infty}
    \Re
    \big(
        (\re^{i\omega_k t} I_n - \re^{tA})
        F(i\omega_k)
    \big)
    Bw_k\\
\nonumber
    & =
    \re^{tA}
    \xi
    +
    \sqrt{\tfrac{2}{T}}
    \sum_{k=0}^{+\infty}
    \Re
    \big(
        \re^{i\omega_k t}
        F(i\omega_k)
    \big)
    Bw_k\\
\label{Xsol2}
    & =
    \re^{tA}\xi
    +
    \sum_{k=0}^{+\infty}
    (
    \omega_k
    f_k(t)
    I_n
    -g_k(t) A
    )
    \mho_k
    B
    w_k,
\end{align}
where $\mho_k$ are the matrices from (\ref{mhok}).
Here, we have also used
the property of the function $F$ in (\ref{XF}) that
$
    F(i\omega) F(-i\omega) = (\omega^2 I_n + A^2)^{-1}
$,
whereby
 $
    F(i\omega) = -(\omega^2 I_n + A^2)^{-1} (A+i\omega I_n)
 $ (and is  well-defined for any $\omega \in \mR$ since $A$ is Hurwitz).
  Also, $\xi$ in (\ref{Xsol2}) is a vector of $n$ self-adjoint quantum variables on the system-field space $\fH$, related to the initial system variables and the QKL coefficients as
$$    \xi
    :=
        X_0
        -
    \sqrt{\tfrac{2}{T}}
    \sum_{k=0}^{+\infty}
    \Re
        F(i\omega_k)
    Bw_k
     =
        X_0
        +
    \sqrt{\tfrac{2}{T}}
    \sum_{k=0}^{+\infty}
    A \mho_k
    B
    w_k,
$$
in accordance with (\ref{xi}).
We will now use the Fourier expansion of the fundamental matrix of the linear system $\dot{x} = Ax$ over the functions (\ref{feig}):
\begin{equation}
\label{exptA}
  \re^{tA}
  =
  I_n
  +
  \sum_{k=0}^{+\infty}
  f_k(t)A_k
  =
  I_n
  +
  \sqrt{\tfrac{2}{T}}
  \sum_{k=0}^{+\infty}
  \sin(\omega_k t)
  A_k
\end{equation}
for all $0\< t \< T$,
with the coefficients $A_k \in\mR^{n\x n}$ computed as
\begin{align}
\nonumber
    A_k
    &=
    \int_0^T
    f_k(t)(\re^{tA}-I_n)
    \rd t\\
\nonumber
    & =
          \sqrt{\tfrac{2}{T}}
    \int_0^T \sin(\omega_k t)(\re^{tA}-I_n)\rd t\\
\nonumber
    & =
          \sqrt{\tfrac{2}{T}}
          \Big(
            \Im
            \big(
                F(-i\omega_k)
                (
                    I_n
                    -
                    \re^{T (i\omega_k I_n + A)}
                )
            \big)
            -
            \tfrac{1}{\omega_k} I_n
          \Big)\\
\nonumber
    & =
      \sqrt{\tfrac{2}{T}}
      \Big(
        \mho_k
        (
                (-1)^k
                A\re^{TA}
                +
                \omega_k I_n
        )
        -
        \tfrac{1}{\omega_k}I_n
      \Big)\\
    & =
      \sqrt{\tfrac{2}{T}}
        A
        \mho_k
        \big(
                (-1)^k
                \re^{TA}
                -
                \tfrac{1}{\omega_k}A
        \big),
\end{align}
in accordance with (\ref{Ak}),
where the matrices $\mho_k$ from (\ref{mhok}) are used together with the identity $\re^{i\omega_k T} = (-1)^k i$ for the frequencies $\omega_k$ in (\ref{feig}). Substitution of (\ref{exptA}) into (\ref{Xsol2}) leads to (\ref{Xfg}), (\ref{ab}).
\end{proof}

The vectors $\alpha_k$, $\beta_k$ in (\ref{ab}) satisfy CCRs whose  structure is more complicated than that of the QKL coefficients $w_k$ in (\ref{wwcomm}) because of the presence of the vector $\xi$ given by (\ref{xi}). More precisely,
\begin{align}
\nonumber
  [\alpha_j,\beta_k^\rT]
  & =
  [A_j \xi+\omega_j \mho_j B w_j,
    ( - A\mho_k Bw_k)^\rT]\\
\nonumber
  & =
  -(A_j [\xi, w_k^\rT] + \omega_j \mho_j B[w_j,w_k^\rT])B^\rT \mho_k^\rT A^\rT\\
\label{abcomm}
  & =
  -2i\Big(\sqrt{\tfrac{2}{T}} A
    A_j \mho_k
 + \delta_{jk} \omega_j\mho_j \Big)BJB^\rT \mho_k^\rT A^\rT.
\end{align}
Here, we have used the commutativity between the initial system variables on $\fH_0$ and the operators on the Fock space $\fF$, whereby $[X_0, w_k^\rT] = 0$ for all $k=0,1,2,\ldots$, which, in view of (\ref{wwcomm}), (\ref{xi}), implies
$$
    [\xi, w_k^\rT]
     =
    [X_0,w_k^\rT]
    +
    \sqrt{\tfrac{2}{T}}
    \sum_{j=0}^{+\infty}
    A \mho_j
    B
    [w_j,w_k^\rT]
     =
    2i
    \sqrt{\tfrac{2}{T}}
    A \mho_k
    B
    J.
$$
In (\ref{abcomm}), the commutativity between the matrices $A$ and $A_k$,  given by  (\ref{Ak}),  has also been used. As opposed to (\ref{wwcomm}), the vectors $\alpha_j$, $\beta_k$, which play the role of coefficients in (\ref{Xfg}),  have a nonvanishing CCR matrix for $j\ne k$. This more complicated commutation structure comes from the fact that Theorem~\ref{th:Xfg} describes the response of the system variables $X_1, \ldots, X_n$ to the QKL expansion of the driving quantum Wiener process $W$, which employs the eigenbasis associated with $W$ rather than $X_1, \ldots, X_n$ themselves.

\section{QUANTUM KARHUNEN-LOEVE REPRESENTATION OF SYSTEM VARIABLES}
\label{sec:QKLX}

Since the matrix $A$ in (\ref{AB}) is assumed to be Hurwitz, then, in the case of vacuum input fields, the system variables $X_1, \ldots, X_n$ of the OQHO have a unique invariant multipoint Gaussian quantum state \cite{VPJ_2018a} with zero mean and the two-point quantum covariance matrix
\begin{equation}
\label{EXX}
  \bE(X(s)X(t)^\rT)
  =
  K(s-t),
  \qquad
  s,t\> 0,
\end{equation}
where
\begin{equation}
\label{K}
    K(\tau)
    =
    \left\{
    \begin{matrix}
    \re^{\tau A}V& {\rm if}\  \tau\> 0\\
    V\re^{-\tau A^{\rT}} & {\rm if}\  \tau< 0\\
    \end{matrix}
    \right.
    =
    K(-\tau)^*,
\end{equation}
with $(\cdot)^*:= {\overline{(\cdot)}}^\rT$ the complex conjugate transpose.
Here,
\begin{equation}
\label{V}
  V:= \Sigma + i\Theta
\end{equation}
is the invariant one-point quantum covariance matrix of the system variables satisfying the ALE
\begin{equation}
\label{VALE}
  AV + VA^\rT + B\Omega B^\rT = 0,
\end{equation}
whose imaginary part is equivalent to the PR condition (\ref{PR}) in view of (\ref{Omegaeig}). Accordingly, the imaginary part of (\ref{K}) describes the two-point CCR matrix
\begin{equation}
\label{XXcomm}
    [X(s), X(t)^\rT]
    =
    2i\Lambda(s-t),
\end{equation}
where
\begin{align}
\nonumber
    \Lambda(\tau)
    & :=
    \Im K(\tau)\\
\label{Lambda}
    & =
    \left\{
    \begin{matrix}
    \re^{\tau A}\Theta & {\rm if}\  \tau\> 0\\
    \Theta \re^{-\tau A^{\rT}} & {\rm if}\  \tau< 0\\
    \end{matrix}
    \right.
    =
    -\Lambda(-\tau)^\rT.
\end{align}
The commutation structure (\ref{XXcomm}), (\ref{Lambda}) of the system variables of the OQHO remains valid regardless of their particular quantum state.

In view of (\ref{VALE}), (\ref{Omegaeig}), the matrix $\Sigma = \Re V$ in (\ref{V}) is the controllability Gramian of the pair $(A,B)$ satisfying the ALE
$$
    A\Sigma + \Sigma A^\rT + BB^\rT = 0.
$$
Being a positive semi-definite Hermitian kernel,
the two-point quantum covariance function $K$ in (\ref{EXX}), (\ref{K}) specifies a positive semi-definite self-adjoint linear integral operator $\cK$ which maps a square integrable function $\varphi: [0,T] \to \mC^n$ to another such function $\psi$ as
\begin{equation}
\label{cK}
    \psi(s)
    :=
    \int_0^T K(s-t)\varphi(t)\rd t,
    \qquad
    0\< s\< T.
\end{equation}
Here, the Hilbert space $L^2([0,T],\mC^n)$ is endowed with the standard inner product
$$
    \bra f,g\ket := \int_0^T f(t)^* g(t)\rd t.
$$
The operator $\cK$,  given by (\ref{K}), (\ref{cK}), is of trace class and its kernel $K$ is represented as
\begin{equation}
\label{Khh}
    K(s-t)
    =
    \sum_{k=0}^{+\infty}
    \mu_k h_k(s) h_k(t)^*,
    \qquad
    0\< s,t \< T,
\end{equation}
in terms of orthonormal eigenfunctions $h_k: [0,T]\to \mC^n$ satisfying
\begin{equation}
\label{Keig}
    \int_0^T K(s-t)h_k(t)\rd t
    =
    \mu_k h_k(s),
    \qquad
    0\< s\< T,\
    k = 0,1,2,\ldots,
\end{equation}
where $\mu_k\>0$ are the corresponding eigenvalues, with $\sum_0^{+\infty} \mu_k = \Tr \cK = T \Tr K(0) = T\Tr \Sigma$, since $\Tr V = \Tr \Sigma$ in (\ref{V}) due to $\Tr \Theta = 0$.
Accordingly,
\begin{equation}
\label{fKf}
    \bra f, \cK g\ket
    =
    \sum_{k=0}^{+\infty}
    \mu_k
    \bra f, h_k\ket  \bra h_k, g\ket,
    \quad
    \bra f, \cK f\ket
    =
    \sum_{k=0}^{+\infty}
    \mu_k
    |\bra f, h_k\ket|^2
\end{equation}
for any $f,g \in L^2([0,T], \mC^n)$. In what follows, we will also use the $\mR^n$-valued functions $\varphi_k:= \Re h_k$, $\psi_k:= \Im h_k$, so that
\begin{equation}
\label{h}
    h_k = \varphi_k + i \psi_k.
\end{equation}
Now, let $\zeta_0, \zeta_1, \zeta_2, \ldots$ be a sequence of vectors
\begin{equation}
\label{zeta}
  \zeta_k
  :=
  \begin{bmatrix}
    \xi_k\\
    \eta_k
  \end{bmatrix}
\end{equation}
which consist of self-adjoint quantum variables $\xi_k$, $\eta_k$ on $\fH$ and satisfy the CCRs
\begin{equation}
\label{zetacomm}
  [\zeta_j,\zeta_k^\rT]
  = 2i \delta_{jk}\bJ,
  \qquad
  j,k=0,1,2,\ldots,
\end{equation}
where the matrix $\bJ$ is given by (\ref{bJ}). Up to a factor of $\sqrt{2}$, the operators $\xi_k$, $\eta_k$ are organised as the quantum mechanical positions and momenta, mentioned in Section~\ref{sec:QKLW}. This gives rise to the annihilation operators
\begin{equation}
\label{gamma}
  \gamma_k := \xi_k + i\eta_k,
\end{equation}
satisfying the CCRs
$$
    [\gamma_j, \gamma_k] = 0,
    \quad
    [\gamma_j^\dagger, \gamma_k^\dagger] = 0,
    \quad
    [\gamma_j, \gamma_k^\dagger] = 4\delta_{jk},
    \quad
    j,k=0,1,2,\ldots,
$$
where $(\cdot)^\dagger $ denotes the operator adjoint.
 Now, consider the series
\begin{align}
\nonumber
    X(t)
    & =
    \sum_{k=0}^{+\infty}
    \sqrt{\mu_k}
    \Re (h_k(t)\gamma_k)\\
\nonumber
    & =
    \sum_{k=0}^{+\infty}
    \sqrt{\mu_k}
    (\varphi_k(t)\xi_k - \psi_k(t)\eta_k)\\
\label{QKLX}
    & =
    \sum_{k=0}^{+\infty}
    \sqrt{\mu_k}
    \begin{bmatrix}
        \varphi_k(t) & - \psi_k(t)
    \end{bmatrix}
    \zeta_k,
    \qquad
    0\< t\< T,
\end{align}
defined in terms of the eigenvalues and eigenfunctions from (\ref{Keig}), (\ref{h})
and the annihilation operators (\ref{gamma}), with the real part extended from complex numbers
to operators as $\Re \xi:= \frac{1}{2}(\xi+\xi^{\dagger})$. It follows from (\ref{zetacomm}), (\ref{QKLX}) that
\begin{align}
\nonumber
    [X(s), X(t)]
    & =
    \sum_{j,k=0}^{+\infty}
    \!\!\sqrt{\mu_j\mu_k}
    {\begin{bmatrix}
        \varphi_j(s) & - \psi_j(s)
    \end{bmatrix}}
    [\zeta_j,\zeta_k^\rT]
    {\begin{bmatrix}
        \varphi_k(t)^\rT \\
         - \psi_k(t)^\rT
    \end{bmatrix}}\\
\nonumber
    & =
    2i
    \sum_{k=0}^{+\infty}
    \mu_k
    {\begin{bmatrix}
        \varphi_k(s) & - \psi_k(s)
    \end{bmatrix}}
    \bJ
    {\begin{bmatrix}
        \varphi_k(t)^\rT \\
         - \psi_k(t)^\rT
    \end{bmatrix}}    \\
\nonumber
    & =
    2i
    \sum_{k=0}^{+\infty}
    \mu_k
    (\psi_k(s)\varphi_k(t)^\rT - \varphi_k(s) \psi_k(t)^\rT)\\
\nonumber
    & =
    2i
    \sum_{k=0}^{+\infty}
    \mu_k
    \Im (h_k(s)h_k(t)^*)\\
\nonumber
    & =
    2i
    \Lambda(s-t),
    \qquad
    0\< s,t\< T,
\end{align}
where use is made of (\ref{bJ}), (\ref{Lambda}), (\ref{Khh}), (\ref{h}). Therefore, (\ref{QKLX}) has the same two-point CCRs (\ref{XXcomm}) as the system variables of the OQHO.

\begin{thm}
\label{th:QKLX}
Suppose the vectors $\zeta_k$ in (\ref{zeta}) with the CCRs (\ref{zetacomm})  are statistically independent and are in a joint Gaussian quantum  state with zero mean and common covariance matrix
\begin{equation}
\label{Gamma}
    \Gamma := I_2 + i\bJ,
\end{equation}
so that
\begin{equation}
\label{Ezzz}
    \bE \zeta_k = 0,
    \quad
    \bE(\zeta_j\zeta_k^\rT) = \delta_{jk}\Gamma,
    \qquad
    j,k=0,1,2,\ldots.
\end{equation}
Then the process (\ref{QKLX}), defined in terms of the eigenvalues and eigenfunctions (\ref{Keig}), (\ref{h}) for the kernel (\ref{K})
and the annihilation operators (\ref{gamma}), has the invariant multipoint Gaussian quantum state of the system variables of the OQHO driven by the vacuum input fields.
\hfill$\square$
\end{thm}
\begin{proof}
The pairwise commutativity of the vectors (\ref{zeta}), their statistical independence and the structure (\ref{Gamma}), (\ref{Ezzz}) of their Gaussian quantum states allow the QCF of the process (\ref{QKLX}) to be computed as
\begin{align}
\nonumber
    \bE \re^{i\int_0^T f(t)^\rT X(t)\rd t}
    & =
    \bE \re^{i\sum_{k=0}^{+\infty}
    \sqrt{\mu_k}
    (\bra f,\varphi_k\ket\xi_k - \bra f,\psi_k\ket\eta_k)}\\
\nonumber
    & =
    \prod_{k=0}^{+\infty}
    \bE \re^{i    \sqrt{\mu_k}
    (\bra f,\varphi_k\ket\xi_k - \bra f,\psi_k\ket\eta_k)}    \\
\nonumber
    & =
    \prod_{k=0}^{+\infty}
    \re^{-\frac{1}{2}\mu_k
    (\bra f,\varphi_k\ket^2+\bra f,\psi_k\ket^2)}    \\
\nonumber
    & =
    \re^{-\frac{1}{2}\sum_{k=0}^{+\infty}\mu_k
    (\bra f,\varphi_k\ket^2+\bra f,\psi_k\ket^2)} \\
\label{XQCF}
    & =
    \re^{-\frac{1}{2}\bra f, \cK f\ket}
\end{align}
for any function $f \in L^2([0,T], \mR^n)$. Here, use is also made of (\ref{fKf}) together with the identity $|\bra f, h_k\ket|^2 = \bra f, \varphi_k\ket^2 +  \bra f, \psi_k\ket^2$ in view of (\ref{h}). The relation (\ref{XQCF}) establishes the multipoint Gaussian quantum state, described in the theorem, for the process $X$ in (\ref{QKLX}).
\end{proof}

The proofs of Theorems~\ref{th:Omegaeig} and \ref{th:QKLX} employ the property that linear transformations of quantum variables in Gaussian states lead to Gaussian quantum variables.
Also note that, regardless of a particular form of the eigenfunctions $h_k$ in (\ref{QKLX}),  the QKL expansion of the system variables (under the conditions of Theorem~\ref{th:QKLX}) is mean square convergent, with the remainder process  $r_N(t):=     \sum_{k=N}^{+\infty}
    \sqrt{\mu_k}
    \Re (h_k(t)\gamma_k)$ satisfying
$
    \int_0^T\bE (r_N(t)^\rT r_N(t)) \rd t
    =
    \sum_{k=N}^{+\infty}
    \mu_k
$ for any $N=0,1,2,\ldots$.  Furthermore, the relatively simple commutation structure and the statistical properties of the coefficients (\ref{zeta}) of the QKL expansion (\ref{QKLX}) of the system variables are similar to those for the QKL expansion (\ref{Weig}) of the quantum Wiener process.

\section{APPLICATION TO QUADRATIC-EXPONENTIAL FUNCTIONALS}
\label{sec:QEF}

For the OQHO, described by (\ref{XCCR})--(\ref{AB}), and assuming the time horizon $T$ to be fixed as before,
consider the following QEF \cite{VPJ_2018a}: 
\begin{equation}
\label{Xi}
    \Xi
    :=
    \bE \re^{Q}.
\end{equation}
Here, $Q$ is a positive semi-definite self-adjoint quantum variable given by 
\begin{align}
\label{Q}
    Q
    :=
    \int_0^T
    X(t)^{\rT} \Pi X(t)
    \rd t,
\end{align}
where
$\Pi$ is a real positive semi-definite symmetric matrix of order $n$. The exponential in (\ref{Xi}) is usually evaluated at $\theta Q$ (instead of $Q$), where   the factor $\theta>0$ is a risk-sensitivity parameter, which is ``absorbed'' here by the matrix $\Pi$. 
The cost functional $\Xi$ imposes an exponential penalty on $Q$ in (\ref{Q}) (which is a quadratic function of the system variables over the time interval $[0,T]$) and involves the mean square cost $\bE Q$ as its limiting case in view of  the asymptotic relation
$$
    \ln \Xi = \bE Q + o(\Pi),
    \qquad
    {\rm as}\
    \Pi \to 0.
$$
The QEF  $\Xi$ in (\ref{Xi}) (when $\theta$ is reinstated and $\Xi=\bE \re^{\theta Q}$ is considered for different $\theta> 0$)  gives rise to an upper bound \cite{VPJ_2018a} for the tail distribution  of the quantum variable $Q$. Furthermore, $\Xi$ also leads to an upper bound for the worst-case value $\sup_{\rho \in \fR}\Tr(\rho Q)$ of the mean square cost $\bE Q$ over a class $\fR$ of those actual density operators $\rho$ whose quantum relative entropy \cite{OW_2010,YB_2009} with respect to the nominal system-field state
\begin{equation}
\label{rho0}
    \rho_0:= \varpi\ox \ups
\end{equation}
does not exceed a given level \cite{VPJ_2018b}. Here, $\varpi$ is the initial system state on the space $\fH_0$, and $\ups$ is the vacuum field state on the Fock space $\fF$. This allows the QEF $\Xi$ to be used as a finite-horizon cost for a closed-loop quantum system, resulting from the connection of a quantum feedback controller and a quantum plant (both modelled as OQHOs), such as in Fig.~\ref{fig:loop}.
\begin{figure}[htbp]
\unitlength=0.9mm
\linethickness{0.4pt}
\begin{picture}(50.00,41.00)
    \put(40,25){\framebox(20,15)[cc]{{}}}
    \put(40,27){\makebox(20,15)[cc]{{\small quantum}}}
    \put(40,23){\makebox(20,15)[cc]{{\small plant}}}

    \put(40,5){\framebox(20,15)[cc]{{}}}
    \put(40,7){\makebox(20,15)[cc]{{\small quantum}}}
    \put(40,3){\makebox(20,15)[cc]{{\small controller}}}

    \put(20,33){\line(1,0){20}}
    \put(15,33){\makebox(0,0)[cc]{$W^{(1)}$}}
    \put(60,33){\line(1,0){10}}

    \put(28,31.1){\line(1,1){4}}
    \put(28,30.9){\line(1,1){4}}
    \put(68,35.1){\line(1,-1){4}}
    \put(68,34.9){\line(1,-1){4}}
    \put(28,15.1){\line(1,-1){4}}
    \put(28,14.9){\line(1,-1){4}}
    \put(68,11.1){\line(1,1){4}}
    \put(68,10.9){\line(1,1){4}}

    \put(70,33){\line(0,-1){25}}
    \put(40,13){\line(-1,0){10}}
    \put(30,13){\line(0,1){25}}
\put(80,13){\line(-1,0){20}}

    \put(85,13){\makebox(0,0)[cc]{$W^{(2)}$}}

\end{picture}\vskip-5mm
\caption{A field-mediated feedback connection of a quantum plant and a quantum controller subject to the augmented  quantum Wiener process $W:=  [{W^{(1)}}^\rT, {W^{(2)}}^\rT]^\rT$, where  $W^{(1)}$, $W^{(2)}$  represent the input fields for the plant and controller, respectively. }
\label{fig:loop}
\end{figure}
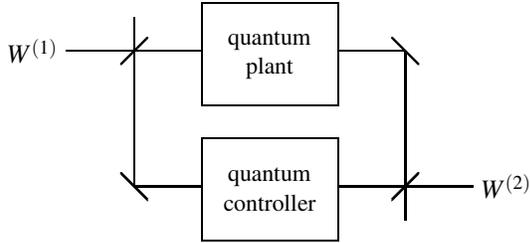
More precisely, even if $\Xi$ is evaluated at the nominal system-field density operator $\rho = \rho_0$ (with $\varpi$ in (\ref{rho0}) being, for example, the invariant Gaussian quantum state for the plant-controller system variables), the minimization of the nominal value of $\Xi$ by an appropriate choice of the controller provides a robust performance criterion for finite-horizon quantum control problems. This makes the development of state-space methods for computing the  QEF an important analysis problem.

For what follows, we assume that  the  OQHO has a Hurwitz matrix $A$ in (\ref{AB}), is driven by vacuum input fields and initialised in the invariant Gaussian quantum state. Then, by Theorem~\ref{th:QKLX},
the QKL series (\ref{QKLX}),  associated with the eigenfunctions of the invariant covariance kernel (\ref{K}), has the invariant multipoint Gaussian quantum state of the system variables. Hence, the QEF in (\ref{Xi}) can be represented by substituting (\ref{QKLX}) into (\ref{Q}):
\begin{align}
\nonumber
  Q
  &=
    \int_0^T
        \sum_{j,k=0}^{+\infty}
    \sqrt{\mu_j\mu_k}
    \zeta_j^\rT
    \begin{bmatrix}
        \varphi_j(t)^\rT \\
        - \psi_j(t)^\rT
    \end{bmatrix}
    \Pi
    \begin{bmatrix}
        \varphi_k(t) & - \psi_k(t)
    \end{bmatrix}
    \zeta_k
    \rd t\\
\label{QX}
& =
    \sum_{j,k=0}^{+\infty}
    \sqrt{\mu_j\mu_k}
    \zeta_j^\rT
    G_{jk}
    \zeta_k,
\end{align}
where
\begin{align}
\nonumber
    G_{jk}
    & :=
      \int_0^T
    \begin{bmatrix}
        \varphi_j(t)^\rT \\
        - \psi_j(t)^\rT
    \end{bmatrix}
    \Pi
    \begin{bmatrix}
        \varphi_k(t) & - \psi_k(t)
    \end{bmatrix}
    \rd t\\
\label{G}
    & =
    \begin{bmatrix}
        \bra \varphi_j, \Pi \varphi_k\ket & -\bra \varphi_j, \Pi \psi_k\ket\\
        -\bra \psi_j, \Pi \varphi_k\ket & \bra \psi_j, \Pi \psi_k\ket
    \end{bmatrix}
    =
    G_{kj}^\rT
\end{align}
are real $(2\x 2)$-matrices consisting of the weighted inner products of the real and imaginary parts of the eigenfunctions (\ref{h}). Therefore, (\ref{QX}) allows the cost functional $\Xi$ in (\ref{Xi}) to be represented as a QEF for the sequence of coefficients $\zeta_k$ of the QKL expansion (\ref{QKLX}) of the continuous-time process $X$:
\begin{equation}
\label{XiQN}
    \Xi
    =
    \bE
    \re^
    {\sum_{j,k=0}^{+\infty}
    \sqrt{\mu_j\mu_k}
    \zeta_j^\rT
    G_{jk}
    \zeta_k}
    =
    \lim_{N\to +\infty}
    \Xi_N,
\end{equation}
where
\begin{equation}
\label{XiN}
    \Xi_N
    :=
    \bE
    \re^{Q_N},
\end{equation}
with
\begin{equation}
\label{QN}
    Q_N
    :=
        {\sum_{j,k=0}^{N-1}
    \sqrt{\mu_j\mu_k}
    \zeta_j^\rT
    G_{jk}
    \zeta_k}
\end{equation}
being a truncation of the infinite series in (\ref{QX}).   In addition to providing a ``meshless discretization'' for the QEF $\Xi$, the representations (\ref{XiQN})--(\ref{QN}) employ quadratic forms of statistically independent Gaussian vectors $\zeta_0, \zeta_1, \zeta_2, \ldots$   which have relatively simple commutation and covariance structures (\ref{zetacomm}), (\ref{Ezzz}).

Now, the computation of the ``incomplete'' QEF $\Xi_N$ in (\ref{XiN}), which involves only a finite number of quantum variables, can be carried out by using the results of \cite[Section~7]{VPJ_2018c}. To this end, (\ref{QN}) is represented for any $N=1,2,3,\ldots$ as
\begin{equation}
\label{QUps}
    Q_N = \Ups_N^\rT H_N \Ups_N,
    \quad
    \Ups_N:= \begin{bmatrix}
      \zeta_0\\
      \vdots\\
      \zeta_{N-1}
    \end{bmatrix}
     = [\xi_0, \eta_0, \ldots, \xi_{N-1},\eta_{N-1}]^\rT,
\end{equation}
where the vector $\Ups_N$ consists of $2N$ self-adjoint quantum variables  from (\ref{zeta}) and satisfies the CCRs
\begin{equation}
\label{Upscomm}
  [\Ups_N, \Ups_N^\rT] = 2i J_N,
  \qquad
  J_N:= I_N \ox \bJ
\end{equation}
in view of (\ref{zetacomm}). Also,
\begin{equation}
\label{HN}
    H_N
    :=
    (\sqrt{\mu_j\mu_k}
    G_{jk})_{0\< j,k< N}
\end{equation}
is a real positive semi-definite symmetric matrix of order $2N$, which is assumed to be nonsingular for what follows.  Since $H_N \succ 0$, Williamson's symplectic diagonalization  theorem \cite{W_1936,W_1937} (see also pp. 244--245 of \cite{D_2006}) guarantees the existence of a symplectic matrix $U_N \in \mR^{2N\x 2N}$ (satisfying $U_N J_N U_N^\rT= J_N$, with the symplectic structure matrix $J_N$ from (\ref{Upscomm})) such that
\begin{equation}
\label{UHU}
    U_N^\rT H_NU_N = S_N\ox I_2,
    \qquad
    S_N
    := \diag_{1\< k\< N}(\sigma_k),
\end{equation}
where $\sigma_1, \ldots, \sigma_N$ are positive real numbers (the symplectic eigenvalues of the matrix $H_N$).  The vector 
$$
  Z_N:= U_N^{-1}\Ups_N
$$
inherits the CCR matrix $J_N$ from $\Ups_N$ in (\ref{QUps}), and, in view of (\ref{Gamma}), (\ref{Ezzz}), its quantum covariance matrix takes the form
\begin{align}
\nonumber
    C_N
    & :=
    \bE(Z_NZ_N^\rT)\\
\nonumber
     & =
    U_N^{-1} (I_{2N} + iJ_N) U_N^{-\rT}\\
\label{CN}
    &
    =
    (U_N^\rT U_N)^{-1} + iJ_N,
\end{align}
where $(\cdot)^{-\rT}:= ((\cdot)^{-1})^\rT$. In order to formulate the theorem below, we associate
\begin{equation}
\label{sympfactk}
    a_k
    =
    \frac{1}{2}
    \tanh (2\sigma_k),
    \quad
   b_k
     =
    \frac{1}{2}
    \sinh(4\sigma_k),
    \quad
    k = 1,\ldots,N,
\end{equation}
with the symplectic spectrum of $H_N$ in (\ref{UHU}), and define auxiliary matrices
\begin{equation}
\label{PhiPsi}
      \Phi_N:=
    I_N
    \ox
    \begin{bmatrix}
      1 & 0\\
      0 & 1\\
      1 & 0
    \end{bmatrix},
    \qquad
    \Psi_N
    :=
    \blockdiag_{1\< k\< N}
    (a_k, b_k, a_k).
\end{equation}
Also, for any matrix $D:= (d_{jk})_{1\< j,k\< s}$, we denote by $D^\diam$ the  matrix of the same order with the entries
\begin{equation}
\label{diam}
(D^\diam)_{jk}
    :=
    \left\{
        \begin{matrix}
            d_{jk} &  {\rm if}\  j \< k\\
            d_{kj} &  {\rm if}\  j > k
        \end{matrix}
    \right.,
\end{equation}
so that $D^\diam$ is a symmetric matrix which inherits its upper triangular part (including the main diagonal) from $D$. The following theorem is established by applying \cite[Theorem~7.1]{VPJ_2018c}.

\begin{thm}
\label{th:QEF}
Suppose the matrix $H_N$ in (\ref{HN}) is positive definite, and its symplectic eigenvalues $\sigma_1, \ldots, \sigma_N$ and the symplectic matrix $U_N$ in (\ref{UHU}) satisfy
\begin{equation}
\label{mhoReL}
    \br((U_N^\rT U_N)^{-1}\blockdiag_{1\< k\< N}(2a_k, b_k)) < 1,
\end{equation}
where $\br(\cdot)$ is the spectral radius. Then the incomplete QEF in (\ref{XiN}) can be computed as
\begin{equation}
\label{XiNex}
    \Xi_N
    =
  \frac{1}{\sqrt{\det(I_{3N} - (\Phi_N C_N\Phi_N^\rT)^{\diam} \Psi_N )}}
\end{equation}
in terms of (\ref{CN})--(\ref{diam}).\hfill$\square$
\end{thm}

The condition (\ref{mhoReL}) (which reflects the ``smallness'' of the matrix $\Pi$ in (\ref{Q}) needed for the QEF $\Xi$ in (\ref{Xi}) and its approximations $\Xi_N$ in (\ref{XiQN}) to be finite) and the representation (\ref{XiNex}) admit a recursive form with respect to $N=1,2,3,\ldots$. In view of (\ref{HN}), application of Theorem~\ref{th:QEF} involves the matrices (\ref{G}) and requires the knowledge of the eigenvalues and eigenfunctions for the invariant covariance kernel (\ref{K}) of the system variables of the OQHO on the time interval $[0,T]$. The eigenanalysis problem (\ref{Keig}) can be tackled by using the matrix exponential structure of the covariance function $K$, which will be discussed elsewhere.

\section{CONCLUSION}
\label{sec:conc}

We have considered a quantum counterpart of the Karhunen-Loeve expansion for the quantum Wiener processes and for system variables of an OQHO, driven by vacuum fields.
A sinusoidal expansion has been obtained for the system variables as their response to the QKL representation of the driving Wiener process. We have also discussed a more natural QKL expansion of the system variables using the eigenvalues and eigenfunctions of the invariant covariance kernel. The common feature of these QKL expansions is the orthonormality of the basis functions and statistical independence of the pairwise commuting Gaussian coefficients each of which consists of conjugate pairs of noncommuting position and momentum operators. We have outlined an application of the QKL representation of the system variables to computing the QEF as a finite-horizon robust performance criterion for linear quantum stochastic systems.

\end{document}